\documentclass[11pt,oneside,english,reqno,A4paper]{amsart}
\usepackage[T1]{fontenc}
\usepackage[latin9]{inputenc}
\usepackage[a4paper]{geometry}
\usepackage{amscd}
\usepackage{amstext}
\usepackage{amsthm}
\usepackage{amssymb}
\usepackage{esint}
\usepackage[unicode=true,pdfusetitle,
 bookmarks=true,bookmarksnumbered=false,bookmarksopen=false,
 breaklinks=false,pdfborder={0 0 1},backref=false,colorlinks=false]
 {hyperref}

\makeatletter
\numberwithin{equation}{section}
\numberwithin{figure}{section}
\theoremstyle{plain}
\newtheorem{thm}{\protect\theoremname}
  \theoremstyle{plain}
  \newtheorem{prop}[thm]{\protect\propositionname}

\usepackage{amsopn}\usepackage{amscd}\usepackage{cite}
\allowdisplaybreaks

\makeatother

\usepackage{babel}
  \providecommand{\propositionname}{Proposition}
\providecommand{\theoremname}{Theorem}

\begin{document}

\title[On asymptotic phase of dynamical system]{On asymptotic phase of dynamical system\\ hyperbolic along attracting invariant manifold}

\author{Alina Luchko and Igor Parasyuk$^{\dag}$}\thanks{$^{\dag}$ National Taras Shevchenko Univesity of Kyiv}

\address{Faculty of Mechanics and Mathematics, Taras Shevchenko National University
of Kyiv, 64/13, Volodymyrska Street, City of Kyiv, Ukraine, 01601 }


\begin{abstract}
We consider a dynamical system which has the hyperbolic structure along an attracting invariant
manifold $M$. The problem is whether every motion starting in a neighborhood
of $M$ possesses an asymptotic phase, i.e. eventually approaches
a particular motion on $M$. Earlier, positive solutions to the problem
were obtained under the condition that the decay rate of solutions
toward the manifold exceeds the decay rate of the solutions within
the manifold. We show that in our case the above condition is not
necessary. To prove that a neighborhood of $M$ is filled with motions
for each of which there exists an asymptotic phase we apply the Brouwer
fixed point theorem. An invariant foliation structure which appears
in the neighborhood of $M$ is discussed.
\end{abstract}

\maketitle

\section{Introduction}

It is well known that under quite general conditions, motions of dissipative
dynamical system evolve towards attracting invariant sets. One may
reasonably expect that the behavior of system on attracting set adequately
displays main asymptotic properties of system motions in the whole
phase space. It is important to note that in many cases the dimension
of attracting set such, e.g., as fixed point, limit cycle, invariant
torus, strange or chaotic attractor, is essentially lower than the
dimension of the total phase space. This circumstance can help us
to simplify the qualitative analysis of the system under consideration.

Nevertheless we should keep in mind that there are cases where no
motion starting outside the attracting invariant set exhibits the
same long time behavior as a motion on the set. As an example consider
the planar system
\begin{gather*}
\begin{cases}
\dot{x}=x(1-x^{2}-y^{2})^{3}-y(1+x^{2}+y^{2}),\\
\dot{y}=x(1+x^{2}+y^{2})+y(1-x^{2}-y^{2})^{3}
\end{cases}
\end{gather*}
 which in polar coordinates $\left(\varphi\bigl|\bmod2\pi,r\right)$
takes the form
\begin{gather*}
\dot{\varphi}=1+r^{2},\quad\dot{r}=r\left(1-r^{2}\right){}^{3}.
\end{gather*}
The limit cycle of the system, $r=1$, attracts all the orbits except
the equilibrium $(0,0)$. Let $\varphi(t;\varphi_{0},r_{0})$ be the
$\varphi$-coordinate of the motion starting at point $(r_{0}\cos\varphi_{0},r_{0}\sin\varphi_{0})$.
Obviously, $\varphi(t;\varphi_{\ast},1)=2t+\varphi_{\ast}$, but if
$r_{0}\not\in\left\{ 0,1\right\} $, then it is not hard to show that
\begin{gather*}
\lim_{t\to\infty}\left|\varphi(t;\varphi_{0},r_{0})-\varphi(t;\varphi_{\ast},1)\right|=\infty\quad\forall\left\{ \varphi_{0},\varphi_{\ast}\right\} \subset[0,2\pi).
\end{gather*}

Let $\left\{ g^{t}(\cdot):\mathfrak{M}\mapsto\mathfrak{M}\right\} _{t\in\mathbb{R}}$
be a flow on a metric space $\left(\mathfrak{M},\rho(\cdot,\cdot)\right)$
with metric $\rho(\cdot,\cdot):\mathfrak{M}\mapsto\mathbb{R}_{+}$,
and let there exists an invariant attracting set $\mathcal{A}\subset\mathfrak{M}$
with a basin $\mathfrak{B}$:
\begin{gather*}
\lim_{t\to\infty}\rho\left(g^{t}(p),\mathcal{A}\right)=0\quad\forall p\in\mathfrak{B}.
\end{gather*}
It is said that a motion $t\mapsto g^{t}(p)$, $p\in\mathfrak{B}$,
has an \emph{asymptotic phase} if there exists $p_{\ast}\in\mathfrak{M}$
such that
\begin{gather*}
\rho\left(g^{t}(p),g^{t}(p_{\ast})\right)\to0,\quad t\to\infty.
\end{gather*}

The following problem arises: what are the conditions guaranteeing
the existence of asymptotic phase? The answer to this problem is rather
important, since the existence of asymptotic phase for any $p\in\mathfrak{B}$
ensures that any motion starting in $\mathfrak{B}$ eventually behaves
like a corresponding motion on $\mathcal{A}$, and thus the flow restricted
to attractor $\mathcal{A}$ faithfully describes the long-time behavior
of the motions starting in $\mathfrak{B}$.

The above problem was studied in a series of papers. The most complete
examination concerns the case where the attracting set is either a
cycle or a manifold fibered by cycles~\cite{Codd_Lev_55,Hartman_64,Chicone_Liu_04,Dumortier_06,Battelli-2011}.
N.~Fenichel~\cite{Fenichel74} established the existence and uniqueness
of asymptotic phase for discrete dynamical system possessing exponentially
stable overflowing invariant manifold with, so-called, expanding structure.
A.~M.~Samojlenko~\cite{Sam_DE_76} and W.~A.~Coppel \cite{Coppel78}
studied the problem for the case of exponentially stable invariant
torus. B B.~Aulbach~\cite{Aulbach_1982} proved the existence of
asymptotic phase for motions approaching a hyperbolic invariant manifold
under assumption that the latter carries a parallel flow. In~\cite{Bog_Ilin_08},
A.~A.~Bogolyubov and Yu.~A.~Il'in estableshed sufficient conditions
ensuring the existence of asymptotic phase for stable invariant torus
(however the authors did not use the notion of asymptotic phase explicitly).
The conditions in~\cite{Bog_Ilin_08} admit non-exponential stability
of invariant torus but exclude the case of exponential divergence
for trajectories within the torus. (See~\cite{Aulbach_1982} for
more comments on the issue).

As was pointed out in~\cite{Aulbach_1982} the conditions ensuring
the existence of an asymptotic phase involve the requirement that
the decay rate of solutions toward the manifold is greater than the
decay rate of the solutions within the manifold. The aim of the present
paper is to show that such a condition is not a necessary one. Like
in~\cite{Fenichel74}, we consider the case of asymptotically stable
hyperbolic invariant manifold, but in contrary to the mentioned article
we deal with a flow rather then a cascade, and besides, in our case,
the decay rate of solutions toward the manifold need not be greater
then the decay rate of the solutions within the manifold. Actually
we exploit the maximal of negative Lyapunov exponents characterizing
the both rates. Our main observation is that one can weaken the expanding
structure condition by abandoning the requirement of asymptotic phase
uniqueness. To prove that a neighborhood of stable invariant manifold
is filled with motions for each of which there exists an asymptotic
phase we apply the Brouwer fixed point theorem, rather then the theorem
on invariance of domain for homeomorphisms as in~\cite{Fenichel74}.

\section{A theorem on the existence asymptotic phase}

Let a $\mathrm{C}^{2}$-vector field $v$ generates the flow$\left\{ \chi^{t}(\cdot):\mathbb{R}^{n}\mapsto\mathbb{R}^{n}\right\} _{t\in\mathbb{R}}$
in space $\mathbb{R}^{n}$ endowed with scalar product $\left\langle \cdot,\cdot\right\rangle $
and norm $\left\Vert \cdot\right\Vert :=\sqrt{\left\langle \cdot,\cdot\right\rangle }$.
Assume that there is a domain $\mathcal{D}\subset\mathbb{R}^{n}$
containing a compact attracting invariant $\mathrm{C^{1}}$-sub-manifold
$\mathcal{M}\hookrightarrow\mathcal{D}$ of dimension $m<n$:
\begin{gather*}
\chi^{t}(\mathcal{M})=\mathcal{M}\quad\forall t\in\mathbb{R};\quad\lim_{t\to\infty}\inf_{\xi\in M}\left\Vert \chi^{t}(x)-\xi\right\Vert =0\quad\forall x\in\mathcal{D}.
\end{gather*}

Consider the autonomous system

\begin{gather}
\dot{x}=v(x).\label{eq:aut-sys-v}
\end{gather}
 From
\begin{gather*}
\frac{\mathrm{d}}{\mathrm{d}t}\frac{\partial\chi^{t}(x)}{\partial x}=v^{\prime}\left(\chi^{t}(x)\right)\frac{\partial\chi^{t}(x)}{\partial x}
\end{gather*}
it follows that
\begin{gather*}
X^{t}(x):=\frac{\partial\chi^{t}(x)}{\partial x}
\end{gather*}
is the normed fundamental matrix of variational system
\begin{gather*}
\dot{y}=v^{\prime}\left(\chi^{t}(x)\right)y,
\end{gather*}
 and the equality $\chi^{t+s}(x)=\chi^{t}\circ\chi^{s}(x)$ implies
the co-cycle property of $X^{t}(x)$:
\begin{gather}
X^{t+s}(x)=X^{t}\left(\chi^{s}(x)\right)X^{s}(x),\quad X^{-s}\left(\chi^{s}(x)\right)=\left[X^{s}(x)\right]^{-1}.\label{eq:X^t-cocycle}
\end{gather}

We say that the flow $\left\{ \chi^{t}(\cdot)\right\} $ is \emph{hyperbolic
along the manifold} $\mathcal{M}$ (equivalently, the manifold $\mathcal{M}$
is said to be\emph{ hyperbolic w.r.t. the flow} $\left\{ \chi^{t}(\cdot)\right\} $)
if:

(1) at every point $x\in\mathcal{M}$, the tangent space $T_{x}\mathbb{R}^{n}\simeq\mathbb{R}^{n}$
is decomposed into the direct sum of three sub-spaces :
\begin{gather}
T_{x}\mathbb{R}^{n}=\mathbb{L}_{x}^{-}\oplus\mathbb{L}_{x}^{+}\oplus\mathbb{L}_{x}^{0},\quad\dim\mathbb{L}_{x}^{\pm,0}=n_{\pm,0},\label{eq:decomp}
\end{gather}
where $\mathbb{L}_{x}^{0}:=\left\{ \lambda v(x)\right\} _{\lambda\in\mathbb{R}}$
is a 1-D subspace spanned by the vector $v(x):=\frac{\mathrm{d}}{\mathrm{d}t}\bigl|_{t=0}\chi^{t}(x)$;

(2) the correspondences $\mathcal{M}\ni x\mapsto\mathbb{L}_{x}^{\pm}$
define continuous and $X^{t}$-invariant fields of planes $\left\{ \mathbb{L}_{x}^{\pm}\right\} _{x\in M}$:
\begin{gather*}
X^{t}(x)\mathbb{L}_{x}^{\pm}=\mathbb{L}_{\chi^{t}(x)}^{\pm}\quad\forall(t,x)\in\mathbb{R}\times\mathcal{M};
\end{gather*}

(3) there exist constants $c\ge1,\alpha>0$ such that
\begin{gather}
\left\Vert X^{t}(x)\eta\right\Vert \le c\mathrm{e}^{-\alpha t}\left\Vert \eta\right\Vert \quad\forall t\ge0,\;\forall x\in\mathcal{M},\;\forall\eta\in\mathbb{L}_{x}^{-},\label{eq:t>0}\\
\left\Vert X^{t}(x)\eta\right\Vert \le c\mathrm{e}^{\alpha t}\left\Vert \eta\right\Vert \quad\forall t\le0,\;\forall x\in\mathcal{M},\;\forall\eta\in\mathbb{L}_{x}^{+},\label{eq:t<0}
\end{gather}

Observe that
\begin{alignat*}{1}
\frac{\mathrm{d}}{\mathrm{d}t}\chi^{t}(x)=v\left(\chi^{t}(x)\right) & \quad\Rightarrow\quad\frac{\mathrm{d}}{\mathrm{d}t}v\left(\chi^{t}(x)\right)=v^{\prime}\left(\chi^{t}(x)\right)v\left(\chi^{t}(x)\right)\\
 & \quad\Rightarrow\quad v\left(\chi^{t}(x)\right)=X^{t}(x)v(x).
\end{alignat*}
 Thus, in addition to~(\ref{eq:t>0}), (\ref{eq:t<0}), we may consider
that
\begin{gather}
\left\Vert X^{t}(x)\eta\right\Vert \le c\left\Vert \eta\right\Vert \quad\forall t\in\mathbb{R},\;\forall x\in\mathcal{M},\;\forall\eta\in\mathbb{L}_{x}^{0}.\label{eq:X^txi_0}
\end{gather}

In what follows, we will consider the case where the sub-manifold
$\mathcal{M}$ is stable, and thus, $\mathbb{L}_{x}^{+}\subset T_{x}\mathcal{M}$
for any $x\in\mathcal{M}$. The space $\mathbb{L}_{x}^{-}$ admits
a decomposition
\begin{gather*}
\mathbb{L}_{x}^{-}=\mathbb{J}_{x}^{-}\oplus\mathbb{K}_{x}^{-}
\end{gather*}
where $\mathbb{K}_{x}^{-}:=T_{x}\mathcal{M}\cap\mathbb{L}_{x}^{-}$
and $\mathbb{J}_{x}^{-}$ is a complement of $\mathbb{K}_{x}^{-}$
in $\mathbb{L}_{x}^{-}$. Hence,
\begin{gather*}
T_{x}\mathbb{R}^{n}=T_{x}\mathcal{M}\oplus\mathbb{J}_{x}^{-},\quad T_{x}\mathcal{M}=\mathbb{K}_{x}^{-}\oplus\mathbb{L}_{x}^{+}\oplus\mathbb{L}_{x}^{0}.
\end{gather*}

Let $T_{\xi}^{\perp}\mathcal{M}$ stands for the orthogonal complement
of tangent space $T_{\xi}\mathcal{M}$ considered as a subspace of
$\mathbb{R}^{n}$. Obviously, $\dim T_{x}^{\perp}\mathcal{M}=\dim\mathbb{J}_{x}^{-}$.
Since $\mathcal{M}$ is compact, then there is sufficiently small
$\delta>0$ such that for any $x\in\mathcal{M}$ the angle between
$\mathbb{J}_{x}^{-}$ and $T_{x}\mathcal{M}$ exceeds $\delta$. This
implies that there is a number $r\in(0,1)$ such that the set
\begin{gather*}
\mathcal{U}_{r}(\mathcal{M}):=\left\{ (\xi,\zeta):\xi\in\mathcal{M},\;\zeta\in\mathbb{J}_{\xi}^{-},\;\left\Vert \zeta\right\Vert <r\right\}
\end{gather*}
 forms a tubular neighborhood of $\mathcal{M}$. Obviously,
\begin{gather*}
\mathcal{U}_{r}(\mathcal{M})\subset\mathcal{N}_{r}(\mathcal{M}):=\left\{ (\xi,\eta):\xi\in\mathcal{M},\;\eta\in\mathbb{L}_{\xi}^{-},\;\left\Vert \eta\right\Vert <r\right\} .
\end{gather*}
The mapping $\mathcal{N}_{r}(\mathcal{M})\ni(\xi,\eta)\mapsto\xi+\eta\in\mathbb{R}^{n}$
define the natural embedding $\mathcal{N}_{r}(\mathcal{M})\hookrightarrow\mathbb{R}^{n}$,
so in what follows we will not distinguish between $(\xi,\eta)\in\mathcal{N}_{r}(\mathcal{M})$
and $\xi+\eta\in\mathbb{R}^{n}$, until it leads to confusion.

Since the field of plains $\left\{ T_{x}\mathcal{M}\right\} _{x\in\mathcal{M}}$
and $\left\{ \mathbb{L}_{x}^{-}\right\} _{x\in\mathcal{M}}$ are $X^{t}$-invariant,
then $\left\{ \mathbb{K}_{x}^{-}\right\} _{x\in\mathcal{M}}$ is $X^{t}$-invariant
as well. In such a case, the flow on invariant manifold, $\left\{ \chi^{t}(\cdot):\mathcal{M}\mapsto\mathcal{M}\right\} _{t\in\mathbb{R}}$,
has the structure of Anosov dynamical system (ADS) (see, e.g., \cite{Arn-Geom_Meth_88}
and Example at the end of this paper). In particular, each of the
fields of planes $\left\{ \mathbb{K}_{x}^{-}\right\} _{x\in\mathcal{M}}$
and $\left\{ \mathbb{L}_{x}^{+}\right\} _{x\in\mathcal{M}}$ are integrable
and form, respectively, the contracting and expanding foliations invariant
w.r.t. the flow $\left\{ \chi^{t}(\cdot)\right\} $. As is well known,
ADSs play an important role in the theory of chaos. Yet another circumstance
that motivate to study the case under consideration is the structural
stability property of ADSs.

Our main result is as follows
\begin{thm}
If the flow $\left\{ \chi^{t}(\cdot)\right\} $ is hyperbolic along
the attracting invariant manifold $\mathcal{M}$, then there is a
neighborhood $\mathcal{U}$ of $\mathcal{M}$ such that any motion
starting in $\mathcal{U}$ has an asymptotic phase.
\end{thm}

The proof of this theorem is based on construction of local contracting
foliation generated by the fields of planes $\left\{ \mathbb{L}_{x}^{-}\right\} _{x\in\mathcal{M}}$.
The existence of such a foliation can be obtained by appropriate interpolation
of corresponding results \cite{Fenichel74} concerning diffeomorphisms.
However, the paper~\cite{Fenichel74} does not contain any details
on the issue.

\section{Proof of the main theorem}

Let $P_{x}^{\pm,0}:\mathbb{R}^{n}\mapsto\mathbb{L}_{x}^{\pm,0}$ be
projections associated with decomposition~(\ref{eq:decomp}). Thus,
$P_{x}^{+}+P_{x}^{-}+P_{x}^{0}=\mathrm{Id}$ where $\mathrm{Id}:\mathbb{R}^{n}\mapsto\mathbb{R}^{n}$
is the identity map in $\mathbb{R}^{n}$. Since for each $x\in\mathcal{M}$
the diagram

\begin{gather*}
\begin{CD}{\mathbb{R}^{n}}@>{X^{t}(x)}>>{\mathbb{R}^{n}}\\
@V{P_{x}^{\pm,0}}VV@VV{P_{\chi^{t}(x)}^{\pm,0}}V\\
{\mathbb{L}_{x}^{\pm,0}}@>{X^{t}(x)}>>{\mathbb{L}_{\chi^{t}(x)}^{\pm,0}}
\end{CD}
\end{gather*}
is commutative, then
\begin{gather}
P_{x}^{\pm,0}=\left[X^{t}(x)\right]^{-1}P_{\chi^{t}(x)}^{\pm,0}X^{t}(x).\label{eq:P_0,pm}
\end{gather}
On account of ~(\ref{eq:P_0,pm}) and~(\ref{eq:X^t-cocycle}) we
easily obtain
\begin{gather*}
X^{t-s}\left(\chi^{s}(x)\right)=X^{t}(x)X^{-s}\left(\chi^{s}(x)\right)=\left[X^{t}(x)\right]\left[X^{s}(x)\right]^{-1}
\end{gather*}
and
\begin{alignat*}{2}
X^{t}(x)P_{x}^{\pm,0}\left[X^{s}(x)\right]^{-1} & =\left[X^{t}(x)\right]\left[X^{s}(x)\right]^{-1}P_{\chi^{s}(x)}^{\pm,0} & =X^{t-s}\left(\chi^{s}(x)\right)P_{\chi^{s}(x)}^{\pm,0}.
\end{alignat*}
If we define $K:=c\max\left\{ \max_{x\in\mathcal{M}}\left\Vert P_{x}^{\pm,0}\right\Vert \right\} $,
then inequalities (\ref{eq:P_0,pm}), (\ref{eq:X^txi_0}) yield
\begin{alignat*}{1}
\left\Vert X^{t}(x)P_{x}^{-}\left[X^{s}(x)\right]^{-1}\right\Vert  & \le K\mathrm{e}^{-\alpha(t-s)},\quad t\ge s,\\
\left\Vert X^{t}(x)P_{x}^{+}\left[X^{s}(x)\right]^{-1}\right\Vert  & \le K\mathrm{e}^{\alpha(t-s)},\quad t\le s.\\
\left\Vert X^{t}(x)P_{x}^{0}\left[X^{s}(x)\right]^{-1}\right\Vert  & \le K,\quad t,s\in\mathbb{R}.
\end{alignat*}

Introduce the new variable $y$ by

\begin{gather*}
x=\chi^{t}(\xi)+y
\end{gather*}
where $\xi\in\mathcal{M}$ is considered as a parameter. If we define
\begin{alignat*}{1}
w(t,y,\xi): & =v\left(\chi^{t}(\xi)+y\right)-v\left(\chi^{t}(\xi)\right)-v^{\prime}\left(\chi^{t}(\xi)\right)y,
\end{alignat*}
then system~(\ref{eq:aut-sys-v}) in new variables takes the form
\begin{gather}
\begin{split}\dot{y} & =v^{\prime}\left(\chi^{t}(\xi)\right)y+w(t,y,\xi).\end{split}
\label{eq:sys-for-y}
\end{gather}
 Since $v$ is $\mathrm{C}^{2}$-vector field, then  there is a constant
$C>0$ such that

\begin{gather}
\left\Vert w(t,y,\xi)\right\Vert \le\frac{C}{2}\left\Vert y\right\Vert ^{2},\quad\left\Vert w_{y}^{\prime}(t,y,\xi)\right\Vert \le C\left\Vert y\right\Vert ,\quad\left\Vert w_{yy}^{\prime\prime}(t,y,\xi)\right\Vert \le C\label{eq:est-w}\\
\forall(t,\xi)\in\mathbb{R}\times\mathcal{M},\;\left\Vert y\right\Vert \le1.\nonumber
\end{gather}

We consider (\ref{eq:sys-for-y}) as a family of systems depending
on parameter $\xi\in\mathcal{M}$. It is not hard to show that, for
a fixed $\xi\in\mathcal{M}$, a mapping $y(\cdot):\mathbb{R}_{+}\mapsto\mathbb{R}^{n}$
is a solution of~(\ref{eq:sys-for-y}) tending to zero as $t\to\infty$
if and only if for some $\eta\in\mathbb{L}_{\xi}^{-}$ the mapping
$y(\cdot)$ satisfies the integral equation
\begin{alignat}{1}
y(t) & =\mathcal{G}[y](t,\xi,\eta):=X^{t}(\xi)\eta+\intop_{0}^{\infty}G(t,s,\xi)w(s,y(s),\xi)\mathrm{d}s\label{eq:int-eqn}
\end{alignat}
with kernel (Green function)
\begin{gather*}
G(t,s,\xi):=\begin{cases}
-X^{t}(\xi)\left[P_{\xi}^{+}+P_{\xi}^{0}\right]\left[X^{s}(\xi)\right]^{-1} & t\le s\\
X^{t}(\xi)P_{\xi}^{-}\left[X^{s}(\xi)\right]^{-1} & t>s.
\end{cases}
\end{gather*}
Besides, for given $\xi\in\mathcal{M}$ and $\eta\in\mathbb{L}_{\xi}^{-}$,
a solution to~(\ref{eq:int-eqn}) satisfies $P_{\xi}^{-}y(0)=\eta$.
\begin{prop}
\label{prop:exist-y}There exist positive numbers $\left\{ r,R\right\} \subset(0,1)$
such that there hold the following assertions:

$(\mathrm{i})$ for each $(\xi,\eta)\in\mathcal{N}_{r}(\mathcal{M})$
there exists a unique solution to~(\ref{eq:sys-for-y}), $y_{\ast}(\cdot,\xi,\eta):\mathbb{R}_{+}\mapsto\mathbb{R}^{n}$,
such that
\begin{equation}
\left\Vert y_{\ast}(t,\xi,\eta)\right\Vert \le R\mathrm{e}^{-\alpha t}\quad\forall t\ge0\quad\textrm{and}\quad P_{\xi}^{-}y(0)=\eta;\label{eq:Y_r,R}
\end{equation}

$(\mathrm{ii})$ the mapping $y_{\ast}(\cdot,\cdot,\cdot):\mathbb{R}_{+}\times\mathcal{N}_{r}(\mathcal{M})\mapsto\mathbb{R}^{n}$
is continuous, and for any $\left(t,\xi\right)\in\mathbb{R}_{+}\times\mathcal{M}$
the mapping $y_{\ast}(t,\xi,\cdot):\mathbb{L}_{\xi}^{-}\mapsto\mathbb{R}^{n}$
is twicely continuous differentiable;

$(\mathrm{iii})$ there exists a constant $C_{0}>0$ such that
\begin{alignat}{1}
\left\Vert y_{\ast}(t,\xi,\eta)-X^{t}(\xi)\eta\right\Vert  & \le C_{0}\mathrm{e}^{-\alpha t}\left\Vert \eta\right\Vert ^{2}\quad\forall(t,\xi,\eta)\in\mathbb{R}_{+}\times\mathcal{N}_{r}(\mathcal{M}).\label{eq:y-X^t-eta}
\end{alignat}
\end{prop}

\begin{proof}
A proof of assertion ($\mathrm{i}$) is obtained in a standard way
by means of the Banach contraction principle. For the sake of completeness
we present here some essential details. Let $\mathrm{C}\left(\mathbb{R}_{+}\!\mapsto\!\mathbb{R}^{n};\alpha\right)$
be the subspace of $\mathrm{C}\left(\mathbb{R}_{+}\!\mapsto\!\mathbb{R}^{n}\right)$
endowed with norm
\[
\left\Vert \cdot\right\Vert _{\infty}:=\sup_{t\ge0}\mathrm{e}^{\alpha t}\left\Vert \cdot\right\Vert .
\]
Define
\begin{gather*}
\mathcal{Y}_{r,R}:=\left\{ y(\cdot)\in\mathrm{C}\left(\mathbb{R}_{+}\!\mapsto\!\mathbb{R}^{n};\alpha\right):\left\Vert y(t)\right\Vert \le R\mathrm{e}^{-\alpha t}\;\forall t\ge0\right\} .
\end{gather*}

Let us impose conditions on $r,R$ under which $\mathcal{G}[\cdot]:\mathcal{Y}_{r,R}\mapsto\mathcal{Y}_{r,R}$.

On account of~(\ref{eq:est-w}), for each $(\xi,\eta)\in\mathcal{N}_{r}(\mathcal{M})$
and each $y(\cdot)\in\mathcal{Y}_{r,R}$, we have the following estimates:

\begin{alignat*}{1}
 & \left\Vert \intop_{t}^{\infty}X^{t}(\xi)\left[P_{\xi}^{+}+P_{\xi}^{0}\right]\left[X^{s}(\xi)\right]^{-1}w(s,y(s),\xi)\mathrm{d}s\right\Vert \\
\le & \frac{1}{2}\intop_{t}^{\infty}K\mathrm{e}^{\alpha(t-s)}CR^{2}\mathrm{e}^{-2\alpha s}\mathrm{d}s+\frac{1}{2}\intop_{t}^{\infty}KCR^{2}\mathrm{e}^{-2\alpha s}\mathrm{d}s\\
\le & \frac{5}{12\alpha}KCR^{2}\mathrm{e}^{-2\alpha t},
\end{alignat*}
\begin{alignat*}{1}
 & \left\Vert \intop_{0}^{t}X^{t}(\xi)P_{\xi}^{-}\left[X^{s}(\xi)\right]^{-1}w(s,y(s),\xi)\mathrm{d}s\right\Vert \\
\le & \frac{1}{2}\intop_{0}^{t}K\mathrm{e}^{-\alpha(t-s)}CR^{2}\mathrm{e}^{-2\alpha s}\mathrm{d}s\le\frac{KCR^{2}}{2\alpha}\mathrm{e}^{-\alpha t}.
\end{alignat*}
 Hence, for all $(t,\xi,\eta)\in\mathbb{R}_{+}\times\mathcal{N}_{r}(\mathcal{M})$
and $y(\cdot)\in\mathcal{Y}_{r,R}$, we obtain
\begin{gather}
\left\Vert \mathcal{G}[y](t,\xi,\eta)\right\Vert \le\left(cr+\frac{11}{12\alpha}KCR^{2}\right)\mathrm{e}^{-\alpha t}\le R\mathrm{e}^{-\alpha t}\label{eq:est-for-GL}
\end{gather}
provided that
\begin{gather}
cr+\frac{11}{12\alpha}KCR^{2}\le R.\label{eq:restrict-1}
\end{gather}

Now let us find conditions under which $\mathcal{G}[\cdot]$ is a
contraction mapping in metric space $\mathcal{Y}_{r,R}$ endowed with
metric $\rho(\star,\ast):=\left\Vert \star-\ast\right\Vert _{\infty}$.
On account of
\begin{alignat*}{1}
 & \left\Vert w(t,y_{1},\xi)-w(t,y_{2},\xi)\right\Vert \\
\le & \left\Vert \intop_{0}^{1}\left[v^{\prime}\left(\chi^{t}(\xi)+sy_{1}+(1-s)y_{2}\right)-v^{\prime}\left(\chi^{t}(\xi)\right)\right]\mathrm{d}s\right\Vert \left\Vert y_{1}-y_{2}\right\Vert \\
\le & \frac{1}{2}C\left(\left\Vert y_{1}\right\Vert +\left\Vert y_{2}\right\Vert \right)\left\Vert y_{1}-y_{2}\right\Vert \quad\forall(t,\xi)\in\mathbb{R}\times\mathcal{M},\quad\left\Vert y_{1}\right\Vert ,\left\Vert y_{2}\right\Vert \le R,
\end{alignat*}
we obtain
\begin{alignat*}{1}
 & \left\Vert \mathrm{e}^{\alpha t}\intop_{t}^{\infty}X^{t}(\xi)P_{\xi}^{+}\left[X^{s}(\xi)\right]^{-1}\left[w(t,y_{1}(s),\xi)-w(t,y_{2}(s),\xi)\right]\mathrm{d}s\right\Vert \\
\le & \mathrm{e}^{\alpha t}KCR\intop_{t}^{\infty}\mathrm{e}^{\alpha(t-s)}\mathrm{e}^{-2\alpha s}\mathrm{d}s\cdot\sup_{t\ge0}\mathrm{e}^{\alpha t}\left\Vert y_{1}(t)-y_{2}(t)\right\Vert \\
\le & \frac{\mathrm{e}^{-\alpha t}}{3\alpha}KCR\rho(y_{1}(\cdot),y_{2}(\cdot));
\end{alignat*}
\begin{alignat*}{1}
 & \left\Vert \mathrm{e}^{\alpha t}\intop_{t}^{\infty}X^{t}(\xi)P_{\xi}^{0}\left[X^{s}(\xi)\right]^{-1}\left[w(t,y_{1}(s),\xi)-w(t,y_{2}(s),\xi)\right]\mathrm{d}s\right\Vert \\
\le & \mathrm{e}^{\alpha t}KCR\intop_{t}^{\infty}\mathrm{e}^{-2\alpha s}\mathrm{d}s\cdot\sup_{t\ge0}\mathrm{e}^{\alpha t}\left\Vert y_{1}(t)-y_{2}(t)\right\Vert \\
\le & \frac{\mathrm{e}^{-\alpha t}}{2\alpha}KCR\rho(y_{1}(\cdot),y_{2}(\cdot));
\end{alignat*}
\begin{alignat*}{1}
 & \left\Vert \mathrm{e}^{\alpha t}\intop_{0}^{t}X^{t}(\xi)P_{\xi}^{-}\left[X^{s}(\xi)\right]^{-1}\left[w(t,y_{1}(s),\xi)-w(t,y_{2}(s),\xi)\right]\mathrm{d}s\right\Vert \\
\le & \mathrm{e}^{\alpha t}KCR\intop_{0}^{t}\mathrm{e}^{-\alpha(t-s)}\mathrm{e}^{-2\alpha s}\mathrm{d}s\cdot\sup_{t\ge0}\mathrm{e}^{\alpha t}\left\Vert y_{1}(t)-y_{2}(t)\right\Vert \\
\le & \frac{1}{\alpha}KCR\rho(y_{1}(\cdot),y(\cdot)).
\end{alignat*}
for all $(t,\xi,\eta)\in\mathbb{R}_{+}\times\mathcal{N}_{r}(\mathcal{M})$
and $y_{1}(\cdot),y_{2}(\cdot)\in\mathcal{Y}_{r,R}$. The above inequalities
implies
\begin{gather*}
\rho\left(\mathcal{G}[y_{1}](\cdot,\xi,\eta),\mathcal{G}[y_{2}](\cdot,\xi,\eta)\right)\le\frac{11}{6\alpha}KCR\rho\left(y_{1}(\cdot),y_{2}(\cdot)\right).
\end{gather*}
Let
\begin{gather*}
\varkappa:=11KCR/(6\alpha)<1.
\end{gather*}
Then by the Banach contraction principle, for each $(\xi,\eta)\in\mathcal{N}_{r}(\mathcal{M})$,
equation~(\ref{eq:int-eqn}) has a unique solution $y_{\ast}(\cdot,\xi,\eta)\in\mathcal{Y}_{r,R}$.
This completes the proof of assertion $(\mathrm{i})$.

As consequence, we have constructed the mapping $y_{\ast}(\cdot,\star,\ast):\mathbb{R}_{+}\times\mathcal{N}_{r}(\mathcal{\mathcal{M}})\mapsto\mathbb{R}^{n}$.
It is easily seen that this mapping is nothing but the unique fixed
point of operator $\mathcal{G}[\cdot]:\tilde{\mathcal{Y}}_{r,R}\mapsto\tilde{\mathcal{Y}}_{r,R}$
where the space $\tilde{\mathcal{Y}}_{r,R}$ consists of mappings
$y(\cdot,\star,\ast)\in\mathrm{C}\left(\mathbb{R}_{+}\times\mathcal{N}_{r}(\mathcal{M})\!\mapsto\!\mathbb{R}^{n}\right)$
satisfying~(\ref{eq:Y_r,R}). A metric in $\tilde{\mathcal{Y}}_{r,R}$
is defined as
\begin{gather*}
\tilde{\rho}(y_{1}(\cdot,\star,\ast),y_{2}(\cdot,\star,\ast)):=\sup\left\{ \mathrm{e}^{at}\left\Vert y_{1}(t,\xi,\eta)-y_{2}(t,\xi,\eta)\right\Vert :(t,\xi,\eta)\in\mathbb{R}_{+}\times\mathcal{N}_{r}(\mathcal{M})\right\} .
\end{gather*}

Let us prove the differentiability of $y_{\ast}(t,\xi,\ast)$. We
will restrict ourselves to the first order derivatives. Let $\xi\in\mathcal{M}$
be fixed at will and let $\mathcal{Y}_{r,R}^{^{\prime}}$ be the space
of continuous mappings $y(\cdot,\ast):\mathbb{R}_{+}\times\mathbb{L}_{\xi}^{-}\mapsto\mathbb{R}^{n}$
having continuous directional derivatives $y_{\eta}^{\prime}(\cdot,\ast)a$
in all directions $a\in\mathbb{L}_{\xi}^{-}$ and satisfying the inequalities
\begin{gather*}
\max\left\{ \left\Vert y(t,\eta)\right\Vert ,\left\Vert y_{\eta}^{\prime}(t,\eta)a\right\Vert \right\} \le R\mathrm{e}^{-\alpha t}\quad\forall(t,\eta)\in\mathbb{R}_{+}\times\mathbb{L}_{\xi}^{-},\quad\forall a\in\mathbb{L}_{\xi}^{-}:\left\Vert a\right\Vert =r.
\end{gather*}

For any $(t,\eta)\in\mathbb{R}_{+}\times\mathbb{L}_{\xi}^{-}$ and
$y(\cdot,\ast)\in\mathcal{Y}_{r,R}^{\prime}$, we obtain

\begin{alignat*}{1}
 & \left\Vert \intop_{t}^{\infty}X^{t}(\xi)\left[P_{\xi}^{+}+P_{\xi}^{0}\right]\left[X^{s}(\xi)\right]^{-1}w_{y}^{\prime}(s,y(s,\eta),\xi)y_{\eta}^{\prime}(s,\eta)a\mathrm{d}s\right\Vert \\
\le & \intop_{t}^{\infty}K\mathrm{e}^{\alpha(t-s)}CR^{2}\mathrm{e}^{-2\alpha s}\mathrm{d}s+\intop_{t}^{\infty}KCR^{2}\mathrm{e}^{-2\alpha s}\mathrm{d}s\\
\le & \frac{5}{6\alpha}KCR^{2}\mathrm{e}^{-2\alpha t},
\end{alignat*}
\begin{alignat*}{1}
 & \left\Vert \intop_{0}^{t}X^{t}(\xi)P_{\xi}^{-}\left[X^{s}(\xi)\right]^{-1}w_{y}^{\prime}(s,y(s,\eta),\xi)y_{\eta}^{\prime}(s,\eta)a\mathrm{d}s\right\Vert \\
\le & \intop_{0}^{t}K\mathrm{e}^{-\alpha(t-s)}CR^{2}\mathrm{e}^{-2\alpha s}\mathrm{d}s\le\frac{KCR^{2}}{\alpha}\mathrm{e}^{-\alpha t}.
\end{alignat*}
Thus,
\begin{gather*}
\left\Vert \mathcal{G}_{\eta}^{\prime}[y](t,\eta)a\right\Vert \le\left(cr+\frac{11}{6\alpha}KCR^{2}\right)\mathrm{e}^{-\alpha t}\le R\mathrm{e}^{-\alpha t}
\end{gather*}
provided that $a\in\mathbb{L}_{\xi}^{-}$, $\left\Vert a\right\Vert =r$
and
\begin{gather}
cr+\frac{11}{6\alpha}KCR^{2}\le R.\label{eq:restrict-2-1}
\end{gather}
The last inequality together with~(\ref{eq:restrict-1}) implies
that $\mathcal{G}[\cdot]:\mathcal{Y}_{r,R}^{\prime}\mapsto\mathcal{Y}_{r,R}^{\prime}$.

Define a metric $\rho^{\prime}$ in $\mathcal{Y}_{r,R}^{\prime}$
as the maximum of
\begin{gather*}
\sup\left\{ \mathrm{e}^{\alpha t}\left\Vert y_{1}(t,\eta)-y_{2}(t,\eta)\right\Vert :(t,\eta)\in\mathbb{R}_{+}\times\mathbb{L}_{\xi}^{-},\left\Vert \eta\right\Vert \le r\right\}
\end{gather*}
and
\begin{gather*}
\sup\left\{ \mathrm{e}^{\alpha t}\left\Vert \frac{\partial y_{1}(t,\eta)}{\partial\eta}a-\frac{\partial y_{2}(t,\eta)}{\partial\eta}a\right\Vert :(t,\eta,a)\in\mathbb{R}_{+}\times\mathbb{L}_{\xi}^{-}\times\mathbb{L}_{\xi}^{-},\left\Vert \eta\right\Vert \le r,\left\Vert a\right\Vert =r\right\} .
\end{gather*}

On account of

\begin{alignat*}{1}
\left\Vert w_{y}^{\prime}(t,y_{1},\xi)u_{1}-w_{y}^{\prime}(t,y_{2},\xi)u_{2}\right\Vert \le & C\left\Vert y_{1}\right\Vert \left\Vert u_{1}-u_{2}\right\Vert +C\left\Vert u_{2}\right\Vert \left\Vert y_{1}-y_{2}\right\Vert \\
\le & C\left(\left\Vert y_{1}\right\Vert +\left\Vert u_{2}\right\Vert \right)\max\left\{ \left\Vert y_{1}-y_{2}\right\Vert ,\left\Vert u_{1}-u_{2}\right\Vert \right\} \\
\forall & (t,\xi)\in\mathbb{R}\times\mathcal{M},\;\left\Vert y_{1}\right\Vert ,\left\Vert y_{2}\right\Vert \le R,\;u_{1},u_{2}\in\mathbb{R}^{n},
\end{alignat*}
in the same way as above we obtain

\begin{gather*}
\mathrm{e}^{\alpha t}\left\Vert \mathcal{G_{\eta}^{\prime}}[y_{1}(\cdot)](t,\eta)a-\mathcal{G_{\eta}^{\prime}}[y_{2}(\cdot)](t,\eta)a\right\Vert \le\frac{11}{3\alpha}KCR\rho^{\prime}\left(y_{1}(\cdot,\ast),y_{2}(\cdot,\ast)\right),
\end{gather*}
 Now to ensure that $\mathcal{G}[\cdot]$ is contracting in ($\mathcal{Y}_{r,R}^{\prime},\rho^{\prime}$),
it is sufficient to impose the condition
\begin{gather}
\frac{11}{3\alpha}KCR<1.\label{eq:restrict-2}
\end{gather}
Finally let us proceed to assertion $(\mathrm{iii})$. Taking into
account the estimates obtained above we have
\begin{alignat*}{1}
\left\Vert \mathrm{e}^{\alpha t}y_{\ast}(t,\xi,\eta)\right\Vert  & \le c\left\Vert \eta\right\Vert +\mathrm{e}^{\alpha t}\intop_{0}^{\infty}\left\Vert G(t,s,\xi)w(t,y(t,\xi,\eta),\xi)\right\Vert \mathrm{d}s\\
 & \le c\left\Vert \eta\right\Vert +\frac{\varkappa}{2}\sup_{t\in\mathbb{R}_{+}}\left\Vert \mathrm{e}^{\alpha t}y_{\ast}(t,\xi,\eta)\right\Vert ,
\end{alignat*}
and thus,
\begin{gather*}
\left\Vert y_{\ast}(t,\xi,\eta)\right\Vert \le\frac{2c}{2-\varkappa}\mathrm{e}^{-\alpha t}\left\Vert \eta\right\Vert \quad\forall(t,\xi,\eta)\in\mathbb{R}_{+}\times\mathcal{\mathcal{N}}_{r}(\mathcal{M}).
\end{gather*}
In its turn, the last inequality together with~(\ref{eq:est-w}),
(\ref{eq:int-eqn}) and
\begin{gather*}
\left\Vert y_{\ast}(t,\xi,\eta)-X^{t}(\xi)\eta\right\Vert \le\frac{C}{2}\intop_{0}^{\infty}\left\Vert G(t,s,\xi)\right\Vert \left\Vert y_{\ast}(s,\xi,\eta)\right\Vert ^{2}\mathrm{d}s
\end{gather*}
implies that there exists a constant $C_{0}>0$ such that there holds
the inequality~(\ref{eq:y-X^t-eta}).
\end{proof}
Now we are in position to prove
\begin{prop}
\label{prop:def-h}Let positive numbers $r,R$ obey the inequalities~(\ref{eq:restrict-2-1})
and~(\ref{eq:restrict-2}). Then there exists a mapping $h(\cdot)\in\mathrm{C}\left(\mathcal{N}_{r}(\mathcal{M})\!\mapsto\!\mathbb{R}^{n}\right)$
such that for each $(\xi,\eta)\in\mathcal{N}_{r}(\mathcal{M})$ there
hold the inequalities $\left\Vert h(\xi,\eta)\right\Vert \le C_{0}\left\Vert \eta\right\Vert ^{2}$
and
\begin{gather*}
\left\Vert \chi^{t}\left(\xi+\eta+h(\xi,\eta)\right)-\chi^{t}(\xi)\right\Vert \le R\mathrm{e}^{-\alpha t}\quad\forall(t,\xi,\eta)\in\mathbb{R}_{+}\times\mathcal{N}_{r}(\mathcal{M}).
\end{gather*}
 Besides $h(\xi,\ast)\in\mathrm{C}^{2}\left(\mathbb{L}_{\xi}^{-}\cap\mathcal{N}_{r}(\mathcal{M})\!\mapsto\!\mathbb{L}_{\xi}^{+}\oplus\mathbb{L}_{\xi}^{0}\right)$
for each $\xi\in\mathcal{M}$.
\end{prop}

\begin{proof}
Define $h(\xi,\eta):=y_{\ast}(0,\xi,\eta)-\eta$. Then
\begin{gather*}
h(\xi,\eta)=-\intop_{0}^{\infty}\left[P_{\xi}^{+}+P_{\xi}^{0}\right]\left[X^{s}(\xi)\right]^{-1}w(s,y_{\ast}(s,\xi,\eta),\xi)\mathrm{d}s\in\mathbb{L}_{\xi}^{+}\oplus\mathbb{L}_{\xi}^{0},
\end{gather*}
and smoothness properties of $h(\xi,\ast)$ follow directly from assertion
(ii) of Proposition~\ref{prop:exist-y}. Since $\chi^{t}(\xi)+y_{\ast}(t,\xi,\eta)$
is the solution of~(\ref{eq:aut-sys-v}) taking value $\xi+\eta+h(\xi,\eta)$
at initial moment of time $t=0$, then
\begin{gather*}
\chi^{t}(\xi)+y_{\ast}(t,\xi,\eta)\equiv\chi^{t}\left(\xi+\eta+h(\xi,\eta)\right).
\end{gather*}
\end{proof}
\begin{prop}
If a positive number $\epsilon\in(0,r/2)$ is sufficiently small,
then for each $(\xi_{0},\zeta_{0})\in\mathcal{U}_{\epsilon}(\mathcal{M})$
there exists $(\xi_{\ast},\zeta_{\ast})\in\mathcal{U}_{2\epsilon}(\mathcal{M})$
such that
\begin{gather*}
\xi_{\ast}+\zeta_{\ast}+h\left(\xi_{\ast},\zeta_{\ast}\right)=\xi_{0}+\zeta_{0}.
\end{gather*}
\end{prop}

\begin{proof}
There exist a domain $Q\subset\mathbb{R}^{m}$ containing the origin
together with its $2\epsilon$-neighborhood and a mapping $\xi(\cdot)\in\mathrm{C}^{2}\left(Q\!\mapsto\!\mathbb{R}^{n}\right)$
such that $\xi(0)=\xi_{0}$ and locally near $\xi_{0}$ the sub-manifold
$\mathcal{M}$ is given by the parametric equation $x=\xi(q)$, $q\in Q$.
One can choose coordinates $(q_{1},\ldots,q_{m})$ in such a way that
column vectors $\left\{ \xi_{q_{i}}^{\prime}(0)\right\} _{i=1}^{m}$
form an orthonormal base of $T_{\xi_{0}}\mathcal{M}$. Let column
vectors $\nu_{1}(q),\ldots,\nu_{n-m}(q)$, where $\nu_{i}(\cdot)\in\mathrm{C}\left(Q\!\mapsto\!\mathbb{R}^{n}\right)$,
$i=1,\ldots,n-m$, form an orthonormal base of $\mathbb{J}_{\xi(q)}^{-}$.
For $z=(z_{1},\ldots,z_{n-m})\in\mathbb{R}^{n-m}$, denote $\nu(q)z:=\sum_{i=1}^{n-m}z_{i}\nu_{i}(q)$.
Observe that $\left\langle \nu_{i}(q),\nu_{j}(q)\right\rangle =\delta_{ij}$
(the Kronecker symbol). Hence, $\left\Vert \nu(q)z\right\Vert ^{2}=\left\Vert z\right\Vert ^{2}$.
Then locally
\begin{alignat*}{1}
\mathcal{U}_{2\epsilon}(\xi(Q))=\bigcup_{q\in Q,\left\Vert z\right\Vert <2\epsilon}\left\{ \xi(q)+\nu(q)z\right\} ,
\end{alignat*}
and there is a unique $z_{0}$ such that $\zeta_{0}=\nu(0)z_{0}$,
$\left\Vert z_{0}\right\Vert <\epsilon$. We have to show that there
exists a solution $\left(q_{\ast},z_{\ast}\right)$ of the equation
\begin{gather*}
\xi(q)+\nu(q)z+h(\xi(q),\nu(q)z)=\xi(0)+\nu(0)z_{0},
\end{gather*}
such that $q_{\ast}\in Q$ and $\left\Vert z_{\ast}\right\Vert <2\epsilon$
. After the change of variable $z=z_{0}+p$ the above equation can
be rewritten as
\begin{gather}
\xi^{\prime}(0)q+\nu(0)p=H(q,p)\label{eq:eqn-for-q,z}
\end{gather}
where
\begin{alignat*}{1}
H(q,p):= & -h\left(\xi(q),\nu(q)\left(z_{0}+p\right)\right)+\intop_{0}^{1}\left[\xi^{\prime}(0)-\xi^{\prime}(sq)\right]q\mathrm{d}s\\
 & +\left[\nu(0)-\nu(q)\right](z_{0}+p).
\end{alignat*}
The mapping $H(\cdot)$ is correctly defined for all $q\in Q$ and
all $p$ such that $\left\Vert p\right\Vert \le\epsilon$, and thus,
the domain of $H(\star,\ast)$ contains the closed ball
\[
\bar{B}_{\epsilon}:=\left\{ (q,p)\in\mathbb{R}^{n}:\left\Vert q\right\Vert ^{2}+\left\Vert p\right\Vert ^{2}\le\epsilon^{2}\right\} .
\]
 Now we are going to show that the existence of solution to~(\ref{eq:eqn-for-q,z})
follows from the Brouwer fixed point theorem.

After introducing the matrix $A$ with columns $\xi_{q_{1}}^{\prime}(0),\ldots,\xi_{q_{m}}^{\prime}(0),\nu_{1}(0),\ldots,\nu_{n-m}(0)$
and column vector $\left(q,p\right)^{\mathrm{T}}=(q_{1},\ldots,q_{m},p_{1},\ldots,p_{n-m})$,
the equation~(\ref{eq:eqn-for-q,z}) is reduced to
\begin{gather}
\left(q,p\right)^{\mathrm{T}}=A^{\mathrm{-1}}H(q,p).\label{eq:eqn-for-q,z-1}
\end{gather}
Since
\begin{gather*}
\mu(\epsilon):=\max_{\left\Vert q\right\Vert \le\epsilon}\left[\left(\intop_{0}^{1}\left\Vert \xi^{\prime}(sq)-\xi^{\prime}(0)\right\Vert \mathrm{d}s\right)^{2}+\left\Vert \nu(q)-\nu(0)\right\Vert ^{2}\right]^{1/2}\to0,\quad\epsilon\to0,
\end{gather*}
then for sufficiently small $\epsilon$ and $(q,p)\in\bar{B}_{\epsilon}$
we obtain
\begin{gather*}
\left\Vert A^{\mathrm{-1}}H(q,p)\right\Vert \le\left\Vert A^{-1}\right\Vert \left[C_{0}\left\Vert z_{0}+p\right\Vert ^{2}+2\mu(\epsilon)\epsilon\right]\le\left\Vert A^{-1}\right\Vert \left[4C_{0}\epsilon^{2}+2\mu(\epsilon)\epsilon\right]\le\epsilon.
\end{gather*}
Hence $A^{-1}H(\star,\ast)\in\mathrm{C}\left(\bar{B}_{\epsilon}\!\mapsto\!\bar{B}_{\epsilon}\right)$,
and by the Brouwer theorem the mapping $A^{-1}H(\star,\ast)$ has
at least one fixed point $(q_{\ast},p_{\ast})\in\bar{B}_{\epsilon}$.
\end{proof}

\section{Invariant foliation}

For the sake of completeness, not pretending on a novelty, we present
here some comments concerning geometric structure which appears in
a neighborhood of $\mathcal{M}$ due to its hyperbolicity.

For each $\xi\in\mathcal{M}$ the set
\begin{gather*}
\mathcal{L}_{\xi}^{-}:=\bigcup_{\eta\in\mathbb{L}_{\xi}^{-}\cap\mathcal{N}_{2\varepsilon}(\mathcal{M})}\left\{ \xi+\eta+h(\eta,\xi)\right\}
\end{gather*}
 is $\mathrm{C}^{2}$-submanifold of $\mathbb{R}^{n}$ diffeomorphic
to $n_{-}$-dimensional ball, provided that $\epsilon$ is sufficiently
small. Since $T_{\xi}\mathcal{L}_{\xi}$=$\mathbb{L}_{\xi}^{-}$,
then $\mathcal{L}_{\xi}^{-}$ transversally intersects $\mathcal{M}$
at point $\xi$. The intersection is a $\mathrm{C}^{1}$-sub-manifold
$\mathcal{M}_{\xi}^{-}$:=$\mathcal{M}\cap\mathcal{L}_{\xi}^{-}$.
A point $\xi^{\prime}\in\mathcal{M}$ from sufficiently small neighborhood
of $\xi$ belongs to $\mathcal{M}_{\xi}^{-}$ iff $\left\Vert \chi^{t}(\xi)-\chi^{t}(\xi^{\prime})\right\Vert =O\left(\mathrm{e}^{-at}\right)$
as $t\to\infty$. Hence, if $\xi^{\prime}\not\in\mathcal{M}_{\xi}^{-}$,
then $\mathcal{M}_{\xi}^{-}\cap\mathcal{M}_{\xi^{\prime}}^{-}=\varnothing$,
and thus, $\mathcal{L}_{\xi}^{-}\cap\mathcal{L}_{\xi^{\prime}}=\varnothing$.
In fact, once we suppose that there is $x_{0}\in\mathcal{L}_{\xi}^{-}\cap\mathcal{L}_{\xi^{\prime}}$
than
\begin{gather*}
\left\Vert \chi^{t}(x_{0})-\chi^{t}(\xi)\right\Vert =O\left(\mathrm{e}^{-at}\right)\quad\textrm{and}\quad\left\Vert \chi^{t}(x_{0})-\chi^{t}(\xi^{\prime})\right\Vert =O\left(\mathrm{e}^{-at}\right)\quad\textrm{as}\;t\to\infty\\
\Rightarrow\quad\left\Vert \chi^{t}(\xi)-\chi^{t}(\xi^{\prime})\right\Vert =O\left(\mathrm{e}^{-at}\right)\quad\textrm{as}\;t\to\infty\quad\Rightarrow\quad\xi^{\prime}\in\mathcal{M}_{\xi}^{-}.
\end{gather*}
Hence, a neighborhood of invariant manifold $\mathcal{M}$ is fibered
with local manifolds $\mathcal{L}_{\xi}^{-}$ continuously dependent
on $\xi$. Each of such manifolds is formed by initial points of motions
associated with common asymptotic phase. Since for each $s\ge0$,
there holds
\begin{gather*}
\left\Vert \chi^{t+s}\left(\xi+\eta+h(\xi,\eta)\right)-\chi^{t+s}(\xi)\right\Vert =O\left(\mathrm{e}^{-\alpha(t+s)}\right),\quad t\to\infty,
\end{gather*}
or, what is the same,
\begin{gather*}
\left\Vert \chi^{t}\circ\chi^{s}\left(\xi+\eta+h(\xi,\eta)\right)-\chi^{t}\circ\chi^{s}(\xi)\right\Vert =O\left(\mathrm{e}^{-\alpha(t+s)}\right),\quad t\to\infty,
\end{gather*}
then $\chi^{s}\left(\xi+\eta+h(\xi,\eta)\right)\in$ $\mathcal{L}_{\chi^{s}(\xi)}^{-}$,
and this implies the invariance of fibers: $\chi^{t}\left(\mathcal{L}_{\xi}^{-}\right)\subset\mathcal{L}_{\chi^{t}(\xi)}^{-}$
for any $t\ge0.$


\begin{thebibliography}{10}

\bibitem{Arn-Geom_Meth_88}
V.~I. Arnol'd.
\newblock {\em Geometrical methods in the theory of ordinary differential
  equations}, volume 250 of {\em Grundlehren der Mathematischen Wissenschaften
  [Fundamental Principles of Mathematical Sciences]}.
\newblock Springer-Verlag, New York, second edition, 1988.

\bibitem{Aulbach_1982}
B.~Aulbach.
\newblock Invariant manifolds with asymptotic phase.
\newblock {\em Nonlinear Anal.}, 6(8):817 -- 827, 1982.

\bibitem{Battelli-2011}
F.~Battelli and K.~J. Palmer.
\newblock Smoothness of asymptotic phase revisited.
\newblock {\em Adv. Nonlinear Stud.}, 11(4):837 -- 851, 2011.

\bibitem{Bog_Ilin_08}
A.~A. Bogolyubov and Y.~A. Il'in.
\newblock Existence of a foliation in a neighborhood of the invariant torus of
  an essentially nonlinear system.
\newblock {\em Differ. Uravn. Protsessy Upr.}, (2):19 -- 38, 2008.

\bibitem{Chicone_Liu_04}
C.~Chicone and W.~Liu.
\newblock Asymptotic phase revisited.
\newblock {\em Journal of Differential Equations}, 204(1):227 -- 246, 2004.

\bibitem{Codd_Lev_55}
E.~A. Coddington and N.~Levinson.
\newblock {\em Theory of ordinary differential equations}.
\newblock McGraw-Hill Book Company, Inc., New York-Toronto-London, 1955.

\bibitem{Coppel78}
W.~A. Coppel.
\newblock {\em Dichotomies in stability theory}.
\newblock Lecture Notes in Mathematics, Vol. 629. Springer-Verlag, Berlin-New
  York, 1978.

\bibitem{Dumortier_06}
F.~Dumortier.
\newblock Asymptotic phase and invariant foliations near periodic orbits.
\newblock {\em Proc. Amer. Math. Soc.}, 134(10):2989 -- 2996, 2006.

\bibitem{Fenichel74}
N.~Fenichel.
\newblock Asymptotic stability with rate conditions.
\newblock {\em Indiana Univ. Math. J.}, 23:1109 -- 1137, 1974.

\bibitem{Hartman_64}
P.~Hartman.
\newblock {\em Ordinary differential equations}.
\newblock John Wiley \& Sons, Inc., New York-London-Sydney, 1964.

\bibitem{Sam_DE_76}
A.~M. {Samojlenko}.
\newblock {Exponential stability of an invariant torus of a dynamic system.}
\newblock {\em {Differ. Equations}}, 11:618 -- 629, 1976.

\end{thebibliography}

\end{document}